\documentclass[a4paper,reqno,10pt]{amsart}

\usepackage{amssymb}
\usepackage{amsmath}
\usepackage{amsthm}
\usepackage{enumitem}
\usepackage{xcolor}
\usepackage{graphicx}
\usepackage{url}
\usepackage{textcomp}
\usepackage{tikz}

\usepackage{fullpage}

\usepackage{todonotes}

\usepackage{texdraw,amstext,amsfonts,color,tabu,dsfont}

\usepackage{mathtools,tabularx}

\usepackage{float}

\definecolor{foge}{rgb}{0.1, 0.6, 0.1}

\numberwithin{equation}{section}

\newtheorem{theo}{Theorem}[section]

\newtheorem{prop}[theo]{Proposition}
\newtheorem{lem}[theo]{Lemma}

\newtheorem{rem}[theo]{Remark}
\newtheorem{ex}[theo]{Example}

\theoremstyle{definition}

\newcommand{\la}{\lambda}

\newcommand{\R}{\mathcal{R}}

\newcommand{\Z}{\mathbb{Z}}

\title{The arithmetical combinatorics of $k,l$-regular partitions}
\author{Isaac Konan}
\address{Universit\'e de Lyon, Universit\'e Claude Bernard Lyon 1, UMR5208, Institut Camille Jordan, F-69622 Villeurbanne, France}
\email{konan@math.univ-lyon1.fr}

\thanks{This work was supported by the LABEX MILYON (ANR-10-LABX-0070) of Universit\'e de Lyon, within the program ''Investissements d'Avenir" (ANR-11-IDEX-0007) operated by the French National Research Agency (ANR).}


\begin{document}
      
\begin{abstract}
For all positive integers $k,l,n$, the Little Glaisher theorem states that the number of partitions of $n$ into parts not divisible by $k$ and occurring less than $l$ times is equal to the number of partitions of $n$ into parts not divisible by $l$ and occurring less than $k$ times. While this refinement of Glaisher theorem is easy to establish by computation of the generating function, there is still no one-to-one canonical correspondence explaining it. Our paper brings an answer to this open problem through an arithmetical approach. Furthermore, in the case $l=2$, we discuss the possibility to construct a Schur-type companion of the Little Glaisher theorem  via the weighted words.
\end{abstract}


\maketitle


\section{Introduction}

An integer partition is a finite non-increasing  sequence of positive integers, called parts of the partition. The weight of an integer partition consists of the sum its parts. In this paper, we enumerate the partitions according to the number of occurrences of positive integers and write the partition in the form $\la = 1^{f_1}2^{f_2}\cdots$. The sequence  $(f_i)_{i\geq 1} $ is called the frequency sequence of $\la$, and the number of occurrences of the part $i$ is referred to as the frequency of $i$.

For  a positive integer $k$, a $k$-regular partition is a partition with no part divisible by $k$, i.e. $f_{ik}=0$ for all $i\geq 1$. This notion of regularity is related to the partition theory, whereas in the group theory, a $k$-regular partition is a partition into parts occurring less than $k$ times, i.e. $f_i<k$ for all $i\geq 1$.
However these two notions of $k$-regularity are closely related. The relation is stated in the following result due to Glaisher \cite{G83}. 

\begin{theo}[Glaisher theorem]\label{theo:glaisher}
At fixed weight, there are as many $k$-regular partitions in terms of partition theory as in terms group theory. The corresponding $q$-series is 
$$
\prod_{k\nmid  i} \frac{1}{1-q^i}=\frac{(q^k;q^k)_\infty}{(q;q)_\infty}= \prod_{i\geq 1} (1+q^i+\cdots+q^{(k-1)i}),
$$
where $(a;q)_n=\prod_{i=0}^{n-1}(1-aq^i)$ for $n\in \Z_{>0}\cup \{\infty\}$.
\end{theo}
In a 2019 paper \cite{KX19}, Keith and Xiong give a refinement of a Sylvester-style bijection \cite{Sy82} of Glaisher's identity due to Stockhofe \cite{Sto82}. We latter generalize this refinement in terms of weighted words in \cite{IK221}.

In the remainder of the paper, we refer to the $k$-regularity as in terms of the partition theory. Back to his original paper, Glaisher gave in a bijection that allows us to link the notion of $k$-regularity to the decomposition of integers in basis-$k$ (see Section \ref{sec:Glaisherbijection}). In this paper, we use a similar approach to prove bijectively an interesting refinement of Glaisher's identity, called the Little Glaisher theorem.
Let $k,l$ be two positive integers. Define a $k,l$-regular partition to be a partition with parts not divisible by $k$ and which occur less than $l$ times, i.e. $f_{ik}=0$ and $f_{i}<l$ for all $i\geq 1$. Denote by $\R_{k,l}$ the set of $k,l$-regular partitions.

\begin{theo}[Little Glaisher theorem]
\label{theo:main}
At fixed weight, there are as many $k,l$-regular partitions as  $l,k$-regular partitions. The corresponding $q$-series is 
$$
\prod_{k\nmid  i} (1+q^i+\cdots+q^{(l-1)i})=\frac{(q^k;q^k)_\infty(q^l;q^l)_\infty}{(q;q)_\infty(q^{kl};q^{kl})_\infty}= \prod_{l\nmid  i} (1+q^i+\cdots+q^{(k-1)i}).
$$
\end{theo}

The remainder of the paper is organized as follows. We first present the Glaisher bijection for Theorem \ref{theo:glaisher}, and show the connection to the decomposition in basis-$k$. After that, in Section \ref{sec:k2}, we show the machinery of our bijection in the case $k=2$, and then, in the same spirit in Section \ref{sec:mainbijection}, we generalize it to prove Theorem \ref{theo:main}.
Finally, in Section \ref{sec:conclusion}, we discuss the possibility to obtain a Schur-type companion using weighted words in the case $l=2$.


\section{The Glaisher bijection}\label{sec:Glaisherbijection}

In this section, we first present the Glaisher bijection show the relation to the decomposition in basis-$k$ of the frequencies of the parts occurring in the $k$-regular partitions.

\subsection{``One should not appear $k$ times''}

The Glaisher bijection is rather simple to implement. In the following, we denote by $\Phi_k$ the Glaisher weight-preserving bijection from the set of $k$-regular partitions to the set of partitions into parts occurring less than $k$ times.

Note that the map is trivial for $k=1$, since the only partition being $1$-regular is the empty partition, and all positive integers have their frequency equal to $0$. We now suppose that $k\geq 2$. 

Let $\la$ be a partition whose frequency sequence $(f_i)_{i\geq 1}$ satisfies $f_{ik}=0$ for all $i\geq 1$. \textbf{As long as} there exists $i\geq 1$ such that $f_i\geq k$, do the transformation
\begin{align}
f_i & \mapsto f_i-k,\label{eq:sub}\\
f_{ik} & \mapsto f_{ik}+1\label{eq:add}.
\end{align} 
Equivalently, this means that $k$ parts equal to $i$ turn into a single part $ik$.
Observe that the iterations stop when $f_i<k$ for all $i\geq 1$, and we set $\Phi_k(\la)$ to be the resulting partition. It may not seem obvious that the choice of $i$ at each step does not affect the final result, but this will be clear once the link to the decomposition in basic-$k$ is established.
 
The inverse bijection is built as follows. For a partition satisfying $f_i<k$ for all $i\geq 1$, as long as there is a positive integer $i$ such that $f_{ik}>0$, do 

\begin{align*}
f_{ik} & \mapsto f_{ik}-1,\\
f_i & \mapsto f_i+k.
\end{align*}

One may notice that the iterations stop when $f_{ik}=0$ for all $i\geq 1$.

\subsection{Dissection in basis-$k$}

We now analyze the transformations occurring during the bijection $\Phi_k$. To clarify the notation, set $f_i(t)$ to be the frequency of the part $i$ after $t$ applications of \eqref{eq:sub},\eqref{eq:add}. Hence, $f_i(0)$ equals $f_i$, the initial frequency of the part $i$ in $\la$. For all $i$ not divisible by $k$, set 
$$S_i(t)=\sum_{h\geq 0} f_{i\cdot k^h}(t) \cdot k^h.$$
Thus, $S_i(0)=f_i$ as $f_{i\cdot k^u} =0$ for all $u\geq 1$.
One may observe that any integer $j\geq 1$ can be uniquely written as a product $i\cdot k^u$ where $k\nmid i$. More precisely, $k^u$ is the largest power of $k$ that divides $j$. Suppose now that the $t^{th}$ transformation turns $m$ parts $j$ into a single part $jk$. Hence, by \eqref{eq:sub},\eqref{eq:add},

\begin{equation}
\label{eq:stept}
\begin{cases}
f_j(t)=f_j(t-1)-k,\\
f_{jk}(t)=f_{jk}(t-1)+1,\\
f_{h}(t)=f_h(t-1) \ \ \text{for} \ \ h\neq j,jk.
\end{cases}
\end{equation}
Writing $j=i\cdot k^u$ with $k\nmid i$, we then have that $S_{g}(t)=S_g(t-1)$ for all $g\neq i$ not divisible by $k$. Moreover, by \eqref{eq:stept},
\begin{align*}
   S_i(t) &= \sum_{h\geq 0} f_{i\cdot k^h}(t) \cdot k^h\\
   &= f_{i\cdot k^u}(t)\cdot k^u+ f_{i\cdot k^u}(t)\cdot k^{u+1} +\sum_{h\neq u, u+1 } f_{i\cdot k^h}(t) \cdot k^h \\
   &= (f_{i\cdot k^u}(t-1)-k)\cdot k^u+ (f_{i\cdot k^u}(t-1)+1)\cdot k^{u+1} +\sum_{h\neq u, u+1 } f_{i\cdot k^h}(t-1) \cdot k^h \\
   &= f_{i\cdot k^u}(t-1)\cdot k^u+ f_{i\cdot k^u}(t-1)\cdot k^{u+1} +\sum_{h\neq u, u+1 } f_{i\cdot k^h}(t-1) \cdot k^h\\
\end{align*}
so that $S_i(t)=S_i(t-1)$.  Hence, for all $i$ not divisible by $k$, $S_i(t)=f_i$ for all $t\geq 0$.
Finally, as 
$$\sum_{i\geq 1} f_i(t) = -k+1 + \sum_{i\geq 1} f_i(t-1),$$
then $$0 \leq \sum_{i\geq 1} f_i(t) = -t(k-1)+\sum_{i\geq 1} f_i$$
so that $t\leq \frac{\sum_{i\geq 1} f_i}{k-1}$. This ensures that the iterations stop at some finite step $T$. Furthermore, for all $k\nmid i$,
$$f_i = S_i(T)=\sum_{h\geq 0} f_{i\cdot k^h}(T) \cdot k^h$$
with $0\leq f_{i\cdot k^h}(T)<k$ for all $h\geq 0$. This is exactly the decomposition of $f_i$ in basis-$k$, which, we recall, is unique. The frequency of $i\cdot k^u$ is  then the coefficient of $k^u$ in basis-$k$ of the frequency of $i$.

Inversely, for a partition $\mu$ whose frequency sequence $(g_i)_{i\geq 1}$ satisfies $g_i<k$ for all $i\geq 0$, the partition $\Phi_k^{-1}(\mu)$ is such that the frequency of $k\nmid i$ equals
$$\sum_{h\geq 0} g_{i\cdot k^h}\cdot k^h,$$
and the frequency of $ik$ equals $0$ for all $i\geq 1$.

\section{The case $k=2$ of the Little Glaisher theorem}\label{sec:k2}

For $k=2$ and $l\geq 1$, Theorem \ref{theo:main} states that, at fixed weight, there are as many partitions into odd parts and occurring less than $l$ times as partitions into distinct parts not divisible by $l$. In the remainder of this section, we use the decomposition $l=2^p\cdot o$ obtained by extracting its largest divisor which is a power of $2$.

\begin{ex}For $l=5,6,8$, we respectively have $l=2^0\cdot 5$, $l=2^1\cdot 3$ and $l=2^3\cdot 1$. 
\end{ex}

We now build a weight-preserving bijection in three steps. 

\subsection*{Step 1}

Let $\la$ be a partition into of odd parts occurring less than $l$ times.
For $i\geq 1$, let $f_{2i-1}$ be the frequency of $2i-1$ and write 
$$f_{2i-1} = \sum_{j=1}^{p+1}\beta_{j}^{(2i-1)}\cdot 2^{j-1}$$
with $\beta_{p+1}^{(2i-1)}=\lfloor f_i/2^p\rfloor$ and the $\beta_1^{(2i-1)},\ldots, \beta_p^{(2i-1)} \in \{0,1\}$ the coefficients in the binary decomposition of $f_i-2^p\lfloor f_i/2^p\rfloor$. Note that $\beta_{p+1}^{(2i-1)}$ is less than $o$. For $j\in \{1,\ldots,p+1\}$, set $\la_{1,j}$ to be the partition consisting of odd multiples of $2^{j-1}$ such that the part $2^{j-1}(2i-1)$ occurs $\beta^{(2i-1)}_j$ times for $i\geq 1$.

Inversely, let $(\la_{1,1},\ldots, \la_{1,p+1})$ be a $p+1$-uplets of partitions such that, for $j\in \{1,\ldots,p+1\}$, the partition  $\la_{1,j}$ consists of odd multiples of $2^{j-1}$ occurring at most once if $j\leq p$ and at most $o-1$ otherwise.
Then, set $\la$ to be the partition into odd part such that $2i-1$ appears $\sum_{j=1}^{p+1}\beta_{j}^{(2i-1)}\cdot 2^{j-1}$ times, where $\beta_{j}^{(2i-1)}$ is the frequency of $2^{j-1}\cdot (2i-1)$ in $\la_{1,j}$. Hence, $\la$ is a partition into odd parts occurring less than $l$ times.

\begin{ex}\label{ex}
Let $l=6=2\cdot 3$, and $\la = 1^23^55^3$. We then have 
$$f_1=0\cdot 1 + 1\cdot 2, \quad f_3 = 1\cdot 1+ 2\cdot 2 \text{ and} \quad f_5=1\cdot 1+1\cdot 2,$$
and obtain $\la_{1,1}=3^15^1$, and $\la_{1,2}=2^16^210^1$.
\end{ex}

\subsection*{Step 2}

Let $\nu_{1,p+1}$ be the partition consisting of the parts of $\la_{1,p+1}$ divided by $2^p$, i.e. the part $2i-1$ occurs $\beta^{(2i-1)}_{p+1}$ for $i\geq 1$. By applying $\Phi_o^{-1}$ on $\nu_{1,p+1}$, we obtain a partition into odd parts not divisible by $o$, since $o$ is odd and an odd $2i-1$ is a multiple of $o$ if and only $(2i-1)/o$ is an odd integer.
Then, by applying $\Phi_2$ to $\Phi_o^{-1}(\nu_{1,p+1})$, we get a partition into distinct parts not divisible by $o$, since a number $i$ is not divisible $o$ if and only $2i$ is not divisible by $o$.

Inversely, by applying $\Phi_o\circ \Phi_2^{-1}$ on any partition into distinct parts not divisible by $o$, we obtain a partition into odd part occurring less than $o$ times.

We finally set $\mu_{1,p+1}$ to be the partitions consisting of multiples of $2^p$ not divisible by $o$, such that, for $i$ not divisible by $o$, the part $2^p i$ occurs as many times as the part $i$ occurs in $\Phi_2(\Phi_o^{-1}(\nu_{1,p+1}))$. The partition $\mu_{1,p+1}$ has then distinct parts divisible by $2^p$ but not divisible by $2^p\cdot o=l$.

We also set $\mu_{1,j}=\la_{1,j}$ for all $j\in \{1,\ldots,p\}$, which consists of distinct odd multiples of $2^{j-1}$, thus not divisible by $2^p\cdot o=l$.
 
\begin{ex}
With the example \ref{ex}, we obtain $\mu_{1,1}=3^15^1$ and $\mu_{1,2}=2^14^18^110^1$.
\end{ex}

\subsection*{Step 3}

The final image is the partition $\mu$ consisting of the parts of all the partitions $\mu_{1,j}$ for $j\in \{1,\ldots,p+1\}$, which we recall are not divisible by $l$. 

Inversely, any partition into distinct parts not divisible by $l$ can be split into $p+1$ partitions by gathering the parts according to the largest power of $2$ in $\{1=2^0,\ldots 2^p\}$ dividing them. 

\begin{ex}
The final image of $\la$ of example \ref{ex} is $\mu=2^13^14^15^18^110^1$.
\end{ex}

\begin{ex}
The full scope of the bijection on the $2,6$-partitions of $10$ is given in the following table:

$$
\begin{array}{|c|c|c|c|}
\hline
\la\in \R_{2,6}& (\la_{1,1},\la_{1,2})& (\mu_{1,1},\mu_{1,2})&\mu\\
\hline\hline
1^19^1&(1^19^1,\emptyset)&(1^19^1,\emptyset)&1^19^1\\
3^17^1&(3^17^1,\emptyset)&(3^17^1,\emptyset)&3^17^1\\
1^37^1&(1^17^1,2^1)&(1^17^1,2^1)&1^12^17^1\\
5^2&(\emptyset,10^1)&(\emptyset,10^1)&10^1\\
1^23^15^1&(3^15^1,2^1)&(3^15^1,2^1)&2^13^15^1\\
1^55^1&(1^15^1,2^2)&(1^15^1,4^1)&1^14^15^1\\
1^13^3&(1^13^1,6^1)&(1^13^1,2^14^1)&1^12^13^14^1\\
1^43^2&(\emptyset,2^26^1)&(\emptyset,2^18^1)&2^18^1\\
\hline
\end{array}\,.
$$ 

\end{ex}

\begin{rem}
For $o=1$, $\la_{1,p+1}=\mu_{1,p+1}=\emptyset$ whatever the choice of the $2,2^p$-regular partition $\la$.
\end{rem}

\section{Bijection for the Little Glaisher theorem}
\label{sec:mainbijection}

In this section, a bijection for Theorem \ref{theo:main} which generalizes the map of Section \ref{sec:k2} is built.
We first connect the Glaisher bijection to the key case where $k$ and $l$ are co-prime. Then, we construct a suitable decomposition of $k,l$-regular partitions based on some arithmetical properties of $k$ and $l$. Finally, we combine this decomposition to the Glaisher bijection and obtain a map which matches bijectively the $k,l$-regular partitions and the $l,k$-regular partitions.

\subsection{The case $\gcd(k,l)=1$} \label{sec:coprime}

In the case $\gcd(k,l)=1$, at fixed weight, the $l,k$-regular partitions are indeed equinumerous to the partitions which are $k$-regular and $l$-regular. 

\begin{prop} 
For $\gcd(k,l)=1$ and a fixed weight, there are as many $l,k$-regular partitions as partitions into parts not divisible neither by $k$ nor by $l$.
\end{prop}

\begin{proof}
Using the Glaisher bijection $\Phi_k$ on a $k$-regular partition, we obtain a partition such that the frequency of $i\cdot k^u$ is the coefficient of $k^u$ in basis-$k$ of the frequency of $i$, for any part $i$ not divisible by $k$.
By the Gauss lemma, $l \mid i\cdot k^{u}$ if and only if $l\mid i$. Moreover, the decomposition in basis-$k$ being unique, the coefficient are all equal to $0$ if and only if the decomposed number is equal to $0$. Hence, the image by $\Phi_k$ is a $l,k$-regular partition if and only the initial $k$-regular partition is also $l$-regular.
\end{proof}

For $\gcd(k,l)=1$, by the above proposition, the map $\Phi_k\circ\Phi_l^{-1}$ defines a weight-preserving bijection from the set of $k,l$-regular partitions to the set of  $l,k$-regular partitions.

\subsection{The case $\gcd(k,l)>1$}

\subsubsection{Preliminaries}

In the remainder of this paper, an empty product is conventionally equal to $1$. In this part, we introduce a decomposition depending of some arithmetical properties of integers. 

\begin{lem}\label{lem:decomposition}
Suppose that the positive integer $d$ can be written as a product $t$ positive integers $d_1,\ldots,d_t$. We then have the following.
\begin{enumerate}
\item The function 
$$(\beta_1,\ldots,\beta_t)\mapsto\sum_{j=1}^{t} \beta_j \prod_{u=1}^{j-1}d_u$$
defines a bijection from the set product 
$$\{0,\ldots,d_1-1\}\times \cdots\times \{0,\ldots,d_j-1\}$$ to $\{0,\ldots,d-1\}$.
\item Any integer $d\nmid i$ can be uniquely written in the form
$$\gamma \prod_{u=0}^{j-1}d_u \text{ for some } 1\leq j \leq t  \text{ and } d_j\nmid \gamma.$$
Inversely, any integer of this form is not divisible by $d$. Hence, 
$$\Z\setminus d\Z = \bigsqcup_{j=1}^t \left(\prod_{u=1}^{j-1}d_u\right) (\Z\setminus d_j\Z).$$
\end{enumerate}
\end{lem}

\begin{proof}
$ $
\begin{enumerate}
    \item We first note that the sets $\{0,\ldots,d_1-1\}\times \cdots\times \{0,\ldots,d_j-1\}$ and $\{0,\ldots,d-1\}$ both have $d$ elements. To prove that the map describes a bijection, it suffices to show that it is surjective. Let $n_1$ be an integer in $\{0,\ldots,d-1\}$. Applying the Euclidean division by $d_1$, one can write $n_1 = \beta_1+d_1\times n_2$, with $\beta_1\in \{0,\ldots,d_1-1\}$ and $0\leq n_2 <d_2\cdots d_t$. Recursively on $1\leq j<t$, if $0\leq n_j <d_j \cdots d_t$, write $n_j = \beta_j + d_j \times n_{j+1}$ with $\beta_j\in \{0,\ldots,d_j-1\}$ and $0\leq n_{j+1}<d_{j+1}\cdots d_t$. Finally, set $\beta_t =n_t$. Then, 
    \begin{align*}
    n_1 &= \beta_1+ d_1\times (\beta_2+d_2\times(\cdots \times(\beta_{t-1}+d_{t-1}\times \beta_t)\cdots)\\
    &= \sum_{j=1}^{t} \beta_j \prod_{u=1}^{j-1}d_u.
    \end{align*}
    This ensures that the map is surjective, and we conclude. 
    
    \item The second statement is straightforward. In fact, as $1=\prod_{u=1}^{0}d_u\mid i$ and $d=\prod_{u=1}^{j}d_u\nmid i$, there exists a unique $1\leq j\leq t$ such that $\prod_{u=1}^{j-1}d_u \mid i$ and $\prod_{u=1}^{j}d_u \nmid i$. Equivalently, $\prod_{u=1}^{j-1}d_u \mid i$ and $d_j \nmid \gamma = i/(\prod_{u=1}^{j-1}d_u)$.
    Conversely, if $d\mid i$, then, for all $1\leq j\leq t$, 
    $$\frac{i}{\prod_{u=1}^{j-1}d_u} = \frac{i}{d}\cdot\prod_{u=j}^t d_u$$
    is divisible by $d_j$. Hence, the only integers that could be written in this form are those not divisible by $d$.
    
\end{enumerate}

\end{proof}

We now write $k=k_1\cdots k_r$ and $l=l_1 \cdots l_s$ in such a way that, for all $1\leq u\leq r$ and $1\leq v \leq s$, either $k_u=l_v$ or $\gcd(k_u,l_v)=1$. Such decomposition of $k$ and $l$ is always possible, the easiest one being the decomposition into primes (it is the least optimal decomposition without factor equal to $1$ in terms of number of factors). 
To ease the notations, note $K_u = \prod_{x=1}^{u-1} k_x$ and $L_v = \prod_{y=1}^{v-1} l_y$ for all $1\leq u\leq r$ and $1\leq v \leq s$. We first decompose our set of $k,l$-regular partitions into a set product of $k_u,l_v$-regular partitions with $1\leq u\leq r$ and $1\leq v \leq s$. 
 
\begin{prop} There is a bijection $\Psi_{k,l}$ from $\R_{k,l}$ to the set product $$\prod_{u=1}^r\prod_{v=1}^s \R_{k_u,l_v}$$
such that, for $\Psi_{k,l}(\la)=(\la_{u,v})_{u,v}$

\begin{equation}\label{eq:weight}
|\la| =  \sum_{u=1}^r \sum_{v=1}^s K_uL_v |\la_{u,v}|,
\end{equation}
where $|\cdot|$ denotes the weight function on the set of partitions.
\end{prop}

\begin{proof}
Let $\la$ be a $k,l$-regular partition with frequency sequence $(f_i)_{i\geq 0}$. Using Lemma \ref{lem:decomposition}, for all $k\nmid i$, as $0\leq f_i<l$, write 
$$f_i = \sum_{v=1}^{s} \beta_v^{(i)} L_v \text{ with } 0\leq \beta_v^{(i)}<l_v \text{ for all } 1\leq v \leq s.$$
For all $1\leq u\leq r$ and $1\leq v \leq s$, set $\la_{u,v}$ to be the partition such that the frequency of $\gamma$ equals $\beta_v^{(\gamma K_u)}$  for all $k_u\nmid \gamma$, and the frequency of any multiple of $k_u$ equals $0$. We then set $\Psi_{k,l}(\la)=(\la_{u,v})_{u,v}$. By fact $(2)$ of Lemma \ref{lem:decomposition},
\begin{align*}
|\la| &= \sum_{k\nmid i} i\cdot f_i\\
&= \sum_{u=1}^r \sum_{k_u\nmid \gamma} \gamma K_u \cdot f_{\gamma K_u}\\
&= \sum_{u=1}^r \sum_{k_u\nmid \gamma} \gamma K_u \sum_{v=1}^s \beta_v^{(\gamma K_u)}L_v\\
 &= \sum_{u=1}^r \sum_{v=1}^s K_uL_v \sum_{k_u\nmid \gamma} \gamma \cdot \beta_v^{(\gamma K_u)}\\
  &= \sum_{u=1}^r \sum_{v=1}^s K_uL_v |\la_{u,v}|,
\end{align*}
and \eqref{eq:weight} holds. Inversely, let $(\la_{u,v})_{u,v}$ be a $k,l$-uplet of partitions  such that $\la_{u,v}$ is a $k_u,l_v$-regular partition for all $1\leq u\leq r$ and $1\leq v \leq s$. For $k_u\nmid \gamma$, let $\eta_{\gamma}^{u,v}$ be the frequency of $\gamma$ in $\la_{u,v}$. Define the partition $\la$ with frequency sequence $(f_i)_{i\geq 1}$ such that 
$$f_{\gamma K_u} = \sum_{v=1}^s \eta_{\gamma}^{u,v} L_v$$
and $f_{ik}=0$ for all $i\geq 1$. As by fact $(2)$ of Lemma \ref{lem:decomposition}, any integer not divisible by $k$ can be uniquely written in the form $\gamma K_u$  with $k_u\nmid \gamma$, we conclude that $\la$ equals $\Phi^{-1}((\la_{u,v})_{u,v})$.

\end{proof}

\subsubsection{The main bijection}

We now build the bijection from $\R_{k,l}$ to $\R_{l,k}$ in three steps.
The first step consists in applying $\Psi_{k,l}$ from $\R_{k,l}$ to $$\prod_{u=1}^r\prod_{v=1}^s \R_{k_u,l_v}.$$
Then, we apply $\Phi_{k_u}\circ \Phi^{-1}_{l_v}$ from $\R_{k_u,l_v}$ to $\R_{l_v,k_u}$ in the set product, and we reach 
$$\prod_{u=1}^r\prod_{v=1}^s \R_{l_v,k_u}.$$ Note that when $k_u=l_v$, $\Phi_{k_u}\circ \Phi^{-1}_{l_v}$ is the identity. Otherwise, $gcd(k_u,l_v)=1$ and Section \ref{sec:coprime} ensures that $\Phi_{k_u}\circ \Phi^{-1}_{l_v}$ defines a bijection from $\R_{k_u,l_v}$ to $\R_{l_v,k_u}$. 
Finally, we apply $\Psi_{l,k}^{-1}$ from  
$$\prod_{u=1}^r\prod_{v=1}^s \R_{l_v,k_u}$$ to 
$\R_{l,k}$. The bijection obtained is then
$$\Psi_{l,k}^{-1}\circ\left(\prod_{u=1}^r\prod_{v=1}^s \Phi_{k_u}\circ \Phi^{-1}_{l_v}\right)\circ\Psi_{k,l}.$$

\begin{ex}
For $k=2$ and $l=2^p\cdot o$ with $o$ odd, the bijection in Section \ref{sec:k2} corresponds to the decomposition $k=k_1=2$ and $l=l_1\cdot l_{p+1}$ with $l_1=\cdots=l_p=2$ and $l_{p+1}=o$.
\end{ex}

\begin{rem} The bijection depends on the decomposition $k=k_1\cdots k_r$ and $l=l_1\cdots l_s$. By permuting the integers in the product, the intermediate phases change and it should be interesting to see whether the final result remains the same or not. However,
by writing $k = \prod_{i\geq 1} p_i^{a_i}$ and $l= \prod_{i\geq 1} p_i^{b_i}$ where $p_i$ runs through prime numbers dividing $kl$, we claim that the optimal decomposition is 
\begin{align*}
k&=\prod_{b_i>0} \left(p_i^{\gcd(a_i,b_i)}\right)^{\frac{a_i}{\gcd(a_i,b_i)}} \times \prod_{b_i=0} p_i^{a_i},\\
 l&=\prod_{a_i>0} \left(p_i^{\gcd(a_i,b_i)}\right)^{\frac{b_i}{\gcd(a_i,b_i)}} \times \prod_{a_i=0} p_i^{b_i},
 \end{align*}
so that $k$ is a product of the integer $\prod_{b_i=0} p_i^{a_i}$ and the integers  $p_i^{\gcd(a_i,b_i)}$ repeated $\frac{a_i}{\gcd(a_i,b_i)}$ times for $b_i>0$, whereas $l$ is a product of the integer $\prod_{a_i=0} p_i^{b_i}$ and the integers  $p_i^{\gcd(a_i,b_i)}$ repeated $\frac{b_i}{\gcd(a_i,b_i)}$ times for $a_i>0$.
\end{rem}

\section{Concluding remarks}\label{sec:conclusion}

We conclude this paper with a discussion on the link between $k,l$-regular partitions and Schur's identity \cite{Sc26}.

\begin{theo}[Schur theorem reformulated]
At fixed weight, there are as many $3,2$-regular partitions as partitions whose frequency sequence $(f_i)_{i\geq 1}$ satisfies the following:

\begin{align*}
f_{i}&+f_{i+1}+f_{i+2}\leq 1,\\
f_{3i}&+f_{3i+1}+f_{3i+2}+f_{3i+3}\leq 1.
\end{align*}

\end{theo}

In this spirit, using the weighted words, Alladi obtained a Schur-like identity related to over-partitions. An over-partition is a partition where each positive integer could occur once as an over-lined part, i.e. a part $\overline{i}$ has frequency $f_{\overline{i}}\in \{0,1\}$.

\begin{theo}[Alladi theorem reformulated]
At fixed weight, there are as many $4,2$-regular partitions as over-partitions whose parts' frequencies satisfy the following:

\begin{align*}
f_{\overline{i}}&=0 \quad \text{for} \quad i=1\text{ or }2\mid i,\\
f_{i}&+f_{\overline{i}}+f_{i+1}+f_{\overline{i+1}}+f_{i+2}+f_{\overline{i+2}}+f_{i+3}+f_{\overline{i+3}}\leq 1 \quad \text{for all}\quad 4\nmid i,\\
f_{i}&+f_{\overline{i}}+f_{i+1}+f_{\overline{i+1}}+f_{i+2}+f_{\overline{i+2}}+f_{i+3}+f_{\overline{i+3}}+f_{i+4}\leq 1 \quad \text{for all}\quad 4\mid i,\\
f_{\overline{i}}&+f_{i+1}+f_{\overline{i+1}}+f_{i+2}+f_{\overline{i+2}}+f_{i+3}+f_{\overline{i+3}}+f_{i+4}+f_{\overline{i+4}}\leq 1 \quad \text{for all}\quad 2\nmid i.
\end{align*}

\end{theo}

This identity can be obtained using a dilatation on a refinement of G\"ollnitz' identity due to  Alladi, Andrews and Gordon \cite{AAG95}. In the framework of this refinement, Alladi, Andrews and Berkovich provide a further generalization of G\"ollnitz' identity in \cite{AAB03}. In Theorem 6 of their paper, by applying the transformation $q,a,b,c,d \mapsto q^5, q^{-4},q^{-3},q^{-2},q^{-1}$ and over-lining the parts  corresponding to the secondary parts except for those colored by $ad$, we derive a Schur-type identity related to the $5,2$-regular partitions. 

\begin{theo}[Alladi--Andrew--Berkovich theorem reformulated]\label{theo:AAB}
At fixed weight, there are as many $5,2$-regular partitions as pairs of partitions $(\mu,\nu)$, where $\mu$ is an over-partitions whose parts' frequencies satisfy 

\begin{align*}
&f_{\overline{i}}=0 \quad \text{for} \quad i=1,2\\
&f_{i}+f_{\overline{i}}+f_{i+1}+f_{\overline{i+1}}+f_{i+2}+f_{\overline{i+2}}+f_{i+3}+f_{\overline{i+3}}+f_{i+4}+f_{\overline{i+4}}\leq 1 \quad \text{for all}\quad i \equiv 2,3 \mod 5 ,\\
&f_{i}+f_{\overline{i-1}}+f_{\overline{i}}+f_{i+1}+f_{\overline{i+1}}+f_{i+2}+f_{\overline{i+2}}+f_{i+3}+f_{\overline{i+3}}+f_{i+4}\leq 1\quad \text{for all}\quad i \equiv 1,4 \mod 5 ,\\
&f_{i}+f_{i+1}+f_{\overline{i}}+f_{\overline{i+1}}+f_{i+2}+f_{\overline{i+2}}+f_{i+3}+f_{\overline{i+3}}+f_{i+4}+f_{\overline{i+4}}+f_{i+5}\leq 1\quad \text{for all}\quad i \equiv 0 \mod 5 ,\\
&f_{\overline{i}}+f_{i+1}+f_{\overline{i+1}}+f_{i+2}+f_{\overline{i+2}}+f_{i+3}+f_{\overline{i+3}}+f_{i+4}+f_{\overline{i+4}}+f_{i+5}+f_{\overline{i+5}}\leq 1 \quad \text{for all } i \equiv 1,2,4 \mod 5,\\
&f_{\overline{i}}+f_{\overline{i+1}}+f_{i+2}+f_{\overline{i+2}}+f_{i+3}+f_{\overline{i+3}}+f_{i+4}+f_{\overline{i+4}}+f_{i+5}+f_{i+6}+f_{\overline{i+5}}\leq 1 \quad \text{for all}\quad i \equiv 0,3 \mod 5,
\end{align*}
and $\nu$ is a partition into parts divisible by $5$ and at least equal to $20+10(\sum_{i\geq 1} f_i+f_{\overline{i}})-\chi(1 \text{ is a part of }\mu)$.
Here, $\chi(A)$ equals $1$ if $A$ is true and $0$ otherwise.
\end{theo}

In a recent paper \cite{IK21}, we provide a bijective proof of the result of Alladi, Andrews and Berkovich using an intermediate identity. By applying the transformation we did for their result, the following companion of Theorem \ref{theo:AAB} derives from Theorem 1.6 of our paper. 

\begin{theo}
At fixed weight, there are as many $5,2$-regular partitions as over-partitions whose parts' frequencies satisfy the following: 

\begin{align*}
&f_{\overline{i}}=0 \quad \text{for} \quad i=1,2\\
&f_{i}+f_{\overline{i}}+f_{i+1}+f_{\overline{i+1}}+f_{i+2}+f_{\overline{i+2}}+f_{i+3}+f_{\overline{i+3}}+f_{i+4}+f_{\overline{i+4}}\leq 1 \quad \text{for all}\quad i \equiv 2,3 \mod 5 ,\\
&f_{i}+f_{\overline{i-1}}+f_{\overline{i}}+f_{i+1}+f_{\overline{i+1}}+f_{i+2}+f_{\overline{i+2}}+f_{i+3}+f_{\overline{i+3}}+f_{i+4}\leq 1\quad \text{for all}\quad i \equiv 1,4 \mod 5 ,\\
&f_{i}+f_{i+1}+f_{\overline{i}}+f_{\overline{i+1}}+f_{i+2}+f_{\overline{i+2}}+f_{i+3}+f_{\overline{i+3}}+f_{i+4}+f_{\overline{i+4}}+f_{i+5}\leq 1\quad \text{for all}\quad i \equiv 0 \mod 5 ,\\
&f_{\overline{i}}+f_{i+1}+f_{\overline{i+1}}+f_{i+2}+f_{\overline{i+2}}+f_{i+3}+f_{\overline{i+3}}+f_{i+4}+f_{\overline{i+4}}+f_{i+5}+f_{\overline{i+5}}\leq 1 \quad \text{for all } i \equiv 1,2,4 \mod 5,\\
&f_{\overline{i}}+f_{\overline{i+1}}+f_{i+2}+f_{\overline{i+2}}+f_{i+3}+f_{\overline{i+3}}+f_{i+4}+f_{i+5}+f_{i+6}+f_{\overline{i+5}}\leq 1 \quad \text{for all}\quad i \equiv 3 \mod 5,\\
&f_{\overline{i}}+f_{\overline{i+1}}+f_{i+2}+f_{\overline{i+2}}+f_{i+3}+f_{\overline{i+3}}+f_{i+4}+f_{\overline{i+4}}+f_{i+6}+f_{\overline{i+5}}\leq 1 \quad \text{for all}\quad i \equiv 0 \mod 5,\\
&f_{5i-1}+f_{\overline{5i+3}}+f_{\overline{5i+7}}\leq 2 \quad \text{for all }  i\geq 1 ,\\
&f_{5i-2}+f_{\overline{5i+3}}+f_{\overline{5i+7}}\leq 2 \quad \text{for all }  i\geq 1 ,\\
&f_{5i-4}+f_{\overline{5i}}+f_{5i+5}\leq 2 \quad \text{for all }  i\geq 2 .\\
\end{align*}
\end{theo}

A further investigation in \cite{IK22} leads to an Alladi--Andrews--Berkovich-type identity for several many primary colors. Hence, in Theorem 1.9, by applying the transformation $q,a_1,\ldots,a_{k-1}\mapsto q^{k},q^{1-k},\ldots,q^{-1}$, one should obtain a Schur-type identity involving $k,2$-regular partitions. Nonetheless, for $k\geq 6$, a explicit enumeration of the partitions satisfying the difference condition is intricate. The difficulty is twofold. First, the obtained parts overlap in terms of congruence modulo $k$ and are not well-ordered in terms of size. In addition, we describe the difference condition in terms of forbidden patterns whose length is not bounded for more than $4$ primary colors. A subsequent research could then consists in finding a suitable approach to describe these partitions.

\end{document}